\documentclass[titlepagea4paper,11pt]{amsart}
\usepackage[arrow,matrix]{xy}
\usepackage{amssymb}
\usepackage{amsbsy,amsmath}
\newtheorem{thm}{Theorem}[section]
\newtheorem{lemma}[thm]{Lemma}
\newtheorem{def+lem}[thm]{Definition+Lemma}
\newtheorem{cor}[thm]{Corollary}

\newtheorem{definition}[thm]{Definition}
\newtheorem{remark}[thm]{Remark}
\newtheorem{prop}[thm]{Proposition}
\newcommand{\Ext}{\mbox{\rm Ext}}

\newcommand{\im}{\mbox{\rm Im}}

\newcommand{\dd}{\partial\bar\partial}

\newcommand{\seq}{\longrightarrow}

\newcommand{\rk}{\mbox{\rm rk\,}}
\newcommand{\m}[1]{\mbox{\rm #1}}
\newcommand{\cal}{\mathcal}
\newcommand{\bb}{\mathbb}

\newtheorem{example}[thm]{Example}

\newenvironment{rem}{\begin{remark}\rm}{\end{remark}}
\newenvironment{defn}{\begin{definition}}{\end{definition}}
\newenvironment{ex}[1]{\begin{example}\rm{\em #1 }}{\hfill$\bigtriangleup$\end{example}}

\title[Rigid rank $2$ vector bundles on non-K\"ahler surfaces]{Rigid holomorphic rank $2$ vector bundles on non-K\"ahler surfaces}

\author{Marco K\"uhnel}
\address{Marco K\"uhnel\\Albert-Ludwigs-Universit\"at Freiburg\\Mathematisches Institut\\Eckerstr. 1\\79104 Freiburg\\Germany}
\email{marco.kuehnel@math.uni-freiburg.de}

\date{\today}

\begin{document}

\begin{abstract}
The interest in rigid vector bundles stems from various sources. From a geometric point of view, non-K\"ahler manifolds are of particular interest with respect to this problem. In this article, a description of the various possible cases of rigid rank $2$ bundles on non-K\"ahler surfaces is presented.
The largest gap of knowledge is for minimal class VII surfaces with $b_{2}(X)\ge 4$.
\end{abstract}

\maketitle

\section{Motivation}

We inquire into determinant preserving deformations of holomorphic vector bundles ${\cal E}\seq X$ over a non-K\"ahler manifold and call ${\cal E}$
rigid if any small determinant preserving deformation of ${\cal E}$ is trivial. The interest in these objects arise from different point of views. Of course, it applies immediately to rigidity problems of the complex structure of projectivisations of vector bundles over non-K\"ahler manifolds.

Another, maybe unexpected, prospective application is the deformation of certain hermitian metrics on open manifolds. If $Y=X\setminus D$ is the complement of a normal crossings divisor in the compact complex manifold $X$, one might look for hermitian metrics on $Y$ with prescribed Ricci curvature and certain asymptotics, e.g. an hermitian metric on $T_{X}(-\log D)$, and ask whether the existence of such a metric is stable under small deformations of $D$. This leads immediately to deformations of holomorphic hermitian vector bundles with fixed Ricci curvature and to the
notion of Ricci-rigidity, describing hermitian vector bundles with only trivial Ricci curvature preserving deformations. The obstruction for the existence of such a deformation lies in $H^1(X,{\cal O}_X)/i_*H^1(X,{\bb R})$. So for K\"ahler manifolds this is unobstructed: Any deformation can be
endowed with an hermitian metric preserving the Ricci curvature. While Ricci curvature preserving
deformations are more abundant than determinant preserving deformations, in general, the notion of Ricci-rigidity turns out to coincide with the usual
rigidity (as defined above).
So the focus switches naturally to non-K\"ahler manifolds with rigid holomorphic vector bundles.
Here we inquire into the case of surfaces and rank $2$ bundles.

Last but not least, this inquiry is motivated by the effort to describe moduli spaces of stable bundles on non-K\"ahler surfaces. 
Stable rigid bundles correspond to point components of certain moduli spaces ${\cal M}^{st}_{\delta, c_{2}}$ of stable rank $2$ vector bundles
with determinant $\delta$ and second Chern class $c_{2}$. Stability is defined with respect to a Gauduchon metric. For class VII surfaces (with small positive $b_{2}(X)$) the moduli space  ${\cal M}^{st}_{K_{X}, 0}$ is of particular interest. Its elements have discriminant $\frac 18b_{2}(X)$. 
For $b_{2}(X)=1$ this moduli space
has been computed in \cite{s08} for all Gauduchon metrics. In this paper we are concerned with the
'extremal' case of discriminant $0$.

\section{Results}
The following theorem states the most important results of this article (cf. Theorems \ref{main1},\ref{main2},\ref{main3}, 
Corollaries \ref{nohopf},\ref{cor3},\ref{cor2}, Propositions \ref{blow},\ref{b2sm4}, Example \ref{exkod}) about existence of rigid rank $2$ bundles on non-K\"ahler compact surfaces.
Here, an Inoue surface is defined as a complex compact surface with $b_{2}(X)=0$ and without curves. 

\begin{thm}\label{result}
Let $X$ be a non-K\"ahler complex compact surface. $X$ allows for a rigid rank $2$ bundle if and only if the minimal model does, and 
then the situation is one of the following list:

\noindent
\begin{tabular}{|c|c|c|c|}
\hline{\bf
list no.}&{\bf $\kappa(X)$}&{\bf minimal model} &{\bf rigid bundles}\\
\hline
1&$-\infty$&Inoue&canonical extension\\
&&&or non-filtrable\\
\hline
2&$-\infty$&$b_{2}(X)\ge 4$&non-filtrable\\
\hline
3&$-\infty$&$b_{2}(X)=3$ w/o cycle&non-filtrable\\
\hline
4&$0$&primary Kodaira&non-filtrable\\
 & &w/ base isogenous to dual fibre& \\
\hline
\end{tabular}
\end{thm}

Some remarks will clarify the picture. 
In case 1 the canonical extension is rigid, 
indeed. There are examples of case 4 with rigid rank $2$ bundles. Case 2 is due to the lack of knowledge about non-filtrable vector bundles.
Case 3 is simply a gap of the classification of class VII surfaces. It is conjectured that there is always a cycle if $b_{2}(X)>0$. The seminal
results of Teleman \cite{te1,te2} about existence of cycles for minimal class VII surfaces with $b_{2}(X)=1,2$ are already incorporated in the list.
By the proof of Theorem \ref{main2} we can even see that for any $b_{2}(X)$ the existence of a cycle excludes rationally 
topologically trivial rigid rank $2$ bundles. An interesting question would be if the lack of rationally topologically trivial rigid rank $2$ bundles already implies the existence of a cycle.

With respect to moduli spaces of stable bundles the following statement is a corollary of Theorem \ref{result}:

\begin{thm}On a non-K\"ahler compact complex surface the stable rigid rank $2$ bundles are exactly the non-filtrable rigid rank $2$ vector bundles.
\end{thm}

According to \cite[Prop. 4.1]{bm05} the moduli space ${\cal M}^{st}_{\delta,c_{2}}$ is smooth and its dimension is zero, if it is non-empty, $\delta^{2}-4c_{2}=0$ and $X$ is a primary Kodaira surface. So we can rephrase implications of Theorem \ref{result} and Lemma \ref{nonfilt} as a smoothness and an existence result.

\begin{cor}Let $X$ be a non-K\"ahler compact complex surface. If ${\cal M}^{st}_{\delta,c_{2}}$ is non-empty and has a zero-di\-men\-sional component,
then it is smooth and has dimension zero.

There are primary Kodaira surfaces admitting a (non-empty) zero-di\-men\-sional smooth moduli space ${\cal M}^{st}_{\delta,c_{2}}$.
\end{cor}

\section{Introduction}

We choose a geometric approach to rigidity. 
Let $X$ be a complex manifold and ${\cal E}$ a holomorphic vector bundle on $X$. A deformation of ${\cal E}$ parametrised by the complex
manifold $T$ with distinguished point $0$ is a holomorphic vector bundle ${\bb E}$ on $X\times T$ such that ${\bb E}|_{X\times\{0\}}\cong\cal E$. We
will write $\cal E_{t}:={\bb E}|_{X\times\{t\}}$. Correspondingly, a small deformation ${\bb E}$ of $\cal E$ is a vector bundle on $X\times(T,0)$, where
$(T,0)$ is the germ of $T$ in $0$. Let $\pi_{1}:X\times(T,0)\seq X$ be the projection onto the first factor. 
A small deformation is trivial if it is isomorphic to $\pi_{1}^{*}\cal E$. It is called determinant preserving if the induced small deformation $\det{\bb E}$ of $\det\cal E$ is trivial.

\begin{defn}
A holomorphic vector bundle $\cal E$ on a complex manifold $X$ is said to be {\em rigid} if every small determinant preserving deformation is trivial. 
\end{defn}

The trace map gives a natural splitting
$${\cal End}({\cal E})={\cal O}_X\oplus\m{ad}({\cal E})$$
into the trace-multiple of the identity and the trace-free part of an endomorphism. Standard deformation theory computes the number
of moduli of ${\cal E}$ with fixed determinant as $h^1(\m{ad}({\cal E}))$ if $H^2(\m{ad}({\cal E}))=0$, i.\,e.\,if ${\cal E}$ is good; in general, there is the Kuranishi family, an effective family of dimension at least $h^1(\m{ad}({\cal E}))-h^2(\m{ad}({\cal E}))$ deforming ${\cal E}$ with fixed determinant \cite[p.302]{FM}.
In the non-K\"ahler setting the discriminant 
$$\Delta({\cal E}):=\frac 1r\left(c_2({\cal E})-\frac{r-1}{2r}c_1({\cal E})^2\right)$$
satisfies
$$\Delta({\cal E})\ge 0$$
for any holomorphic vector bundle ${\cal E}$ of rank $r$ (cf. \cite{brf},\cite{bp87}).
So the following formula -- simply a reformulation of the Riemann-Roch formula -- proves to be useful. On surfaces $X$ we have
\begin{equation}\label{fam}h^1(\m{ad}({\cal E}))-h^2(\m{ad}({\cal E}))=h^0(\m{ad}({\cal E}))+2r^2\Delta({\cal E})-(r^2-1)\chi({\cal O}_X). 
\end{equation}
Note that any non-K\"ahler surface $X$ satisfies $\chi({\cal O}_{X})=0$.
This implies immediately

\begin{prop}\label{imm}
Let $X$ be a compact complex non-K\"ahler surface 
and ${\cal E}$ a holomorphic vector bundle on $X$. If ${\cal E}$ is
rigid, then ${\cal E}$ is simple and $\Delta({\cal E})=0$. 
\end{prop}

This triggers the question about stability of rigid bundles. In this context, stability is understood with respect to a Gauduchon metric $g$ 
on $X$, which induces a degree
$$\deg_{g}({\cal F}):=\int_{X}c_{1}({\cal F}, h)\wedge\omega,$$
where $c_{1}({\cal F}, h)$ is the first Chern form with respect to the Chern connection of any hermitian metric $h$ on a locally free sheaf
${\cal F}$ and $\omega$ is the fundamental form of the Gauduchon metric $g$. However, it follows from Prop. 5.2.1 in \cite{lt} that for vector bundles
${\cal E}$ with $\Delta({\cal E})=0$ stability is actually independent of the Gauduchon metric $g$. At this point, more cannot be said. We will see later
an example of a rigid, unstable rank $2$ bundle.
 
Consider a blowup $\pi:Y\seq X$ of a point $x$ with exceptional divisor $E$ and a Gauduchon metric $g$ on $X$ with fundamental form $\omega$.
If $\eta$ is the $(1,1)$-form corresponding to the Poincar\'e dual of $E$, then
$$\tilde\omega_{\varepsilon}:=\pi^{*}\omega-\varepsilon\eta$$
is the fundamental form of a Gauduchon metric $\tilde g_{\varepsilon}$ on $Y$ for small enough $\varepsilon$.  
How do rigidity and stability of vector bundles behave towards blowup? In the sequel we have to distinguish between filtrable and non-filtrable
rigid vector bundles. 

\begin{prop}\label{blow}
 Let $X$ be a compact complex surface and $\pi:Y\seq X$ be the blowup of $X$ in a point. 
\begin{enumerate}
 \item A holomorphic rank $2$ vector bundle ${\cal E}$ on  $X$ is filtrable if and only if $\pi^*{\cal E}$ is.
 \item A holomorphic rank $2$ vector bundle ${\cal E}$ on  $X$ is $g$-stable if $\pi^*{\cal E}$ is $\tilde g_{\varepsilon}$-stable.
 \item A holomorphic vector bundle ${\cal E}$ on  $X$ is rigid if and only if $\pi^*{\cal E}$ is.
 \item If $X$ is non-K\"ahler and $\tilde{\cal E}$ on $Y$ a  rigid holomorphic rank $2$ bundle, then there is a 
rigid holomorphic rank $2$ bundle ${\cal E}$ on $X$ such that $\tilde{\cal E}=\pi^*{\cal E}$.	
\end{enumerate}
\end{prop}

\begin{proof}\begin{enumerate}
\item  Let $E\cong{\bb P}^1$ be the exceptional divisor
of $\pi$ and $x$ the blown up point. Let us assume $\pi^*{\cal E}$
is filtrable, i.\,e.\,there is a line bundle
${\cal L}$ on $Y$ such that $H^0(\pi^*\cal E\otimes{\cal L})\not=0$. We decompose
$${\cal L}=\pi^*{\cal M}\otimes{\cal O}_Y(kE)$$
for some line bundle ${\cal M}$ on $X$ and $k\in{\bb Z}$. We denote
$${\cal S}_{k}:=\pi_{*}{\cal O}_Y(kE)=\left\{\begin{array}{cc}{\cal O}_{X}&\text{, if }k\ge 0\\
{\cal J}_{x}^{-k}&\text{, if }k<0\end{array}\right.$$
and compute by pushing forward
$$0\not=H^0(\pi^*\cal E\otimes{\cal L})=H^0({\cal E}\otimes{\cal M}\otimes{\cal S}_{k})\subset H^0({\cal E}\otimes{\cal M}).$$
In particular, ${\cal E}$ is filtrable.
The converse is trivially true.

\item Assume $\pi^*{\cal E}$ is $\tilde g_{\varepsilon}$-stable. Let ${\cal F}$ be a subline bundle of ${\cal E}$. Then
$$\deg_{g}{\cal F}=\deg_{\tilde g_{\varepsilon}}\pi^{*}{\cal F}<\frac 12\deg_{\tilde g_{\varepsilon}}\pi^{*}\cal E=\deg_{g}\cal E,$$
i.\,e.\,$\cal E$ is $g$-stable.

\item Let $\tilde{\cal E}_t$ be a small deformation of $\pi^*{\cal E}$ preserving $\pi^*\det{\cal E}$. Since
$$\pi^*{\cal E}|_E={\cal O}_E^{\oplus r},$$
also
$$\tilde{\cal E}_t|_E={\cal O}_E^{\oplus r}.$$
This again means that $\tilde{\cal E}_t=\pi^*{\cal E}_t$ for the vector bundle ${\cal E}_t:=\pi_*\tilde{\cal E}_t$.
So the deformation was induced by a deformation ${\cal E}_t$ of ${\cal E}$. Obviously, triviality of $\tilde{\cal E}_t$
is equivalent to the triviality of ${\cal E}_t$. This proves that $\pi^*{\cal E}$ is rigid if ${\cal E}$ is.
The converse is trivial.

\item We prove that $\Delta(\tilde{\cal E})=0$ implies $\tilde{\cal E}=\pi^{*}{\cal E}$.  After tensoring with a line bundle we may assume a
splitting
$$\tilde{\cal E}|_{E}={\cal O}_{E}\oplus{\cal O}_{E}(-\lambda)$$
for some $\lambda\ge 0$ and assume $\lambda>0$. The splitting sequence
$$0\seq{\cal O}_{E}\seq\tilde{\cal E}|_{E}\seq{\cal O}_{E}(-\lambda)\seq 0$$
gives rise to an elementary modification $\tilde{\cal F}$ of $\tilde{\cal E}$ with discriminant
$$\Delta(\tilde{\cal F})=\Delta(\tilde{\cal E})+\frac 18(1-2\lambda)=\frac 18(1-2\lambda)<0,$$
a contradiction. So $\lambda=0$ and by the criterion mentioned above and the second claim of this proposition 
this means $\tilde{\cal E}=\pi^{*}{\cal E}$ for a rigid rank $2$ bundle ${\cal E}$ on $X$.

             \end{enumerate}
\end{proof}

\begin{rem}
We could prove here directly that on non-K\"ahler surfaces with $a(X)=0$ stability of a rigid rank $2$ bundle on $X$ implies also the stability
of the pullback under a blowup. However, as it will turn out that the only stable rigid rank $2$ bundles are non-filtrable on any non-K\"ahler surface,
a more general claim follows then already by Prop.\,\ref{blow}. This proposition will also be used to prove the non-filtrability of stable rigid rank $2$ bundles.\end{rem}

In the following it will be important to distinguish between filtrable and non-filtrable rigid rank $2$ bundles. This is based on the following result.

\begin{lemma}\label{nonfilt}
Let $X$ be a non-K\"ahler compact complex surface with $\kappa(X)=\{-\infty,0\}$, but the minimal model $X'$
not a secondary Kodaira surface with $K_{X'}$ of order $2$,
and ${\cal E}$ a non-filtrable holomorphic rank $2$ bundle over
$X$. Then ${\cal E}$ is rigid if and only if $\Delta({\cal E})=0$. Moreover, ${\cal E}$ is good if it is rigid.
\end{lemma}

\begin{proof}
We assume $\Delta({\cal E})=0$ and by Proposition \ref{blow} we may assume that
$X$ is minimal.
Since any non-filtrable vector bundle is simple, equation (\ref{fam}) implies that it suffices to prove $H^2(\m{ad}({\cal E}))=0$.
If $K_{X}={\cal O}_{X}$, then the claim follows directly by the simplicity of ${\cal E}$.
If $K_{X}\not={\cal O}_{X}$, this is equivalent to $H^2({\cal End}({\cal E}))=0$.  Using Serre duality we want to prove that any vector bundle homomorphism
$$s:{\cal E}\seq{\cal E}\otimes K_X$$
is zero. We observe $\det s\in H^0(K_X^{\otimes 2})=0,$ by our assumptions. If $s\not\equiv 0$, then $\ker s\subset{\cal E}$ is a coherent rank $1$ subsheaf. So ${\cal E}$ is filtrable, contrary
to the assumption. This finally proves $s\equiv 0$.  
\end{proof}

We will make extensive use of Lemma \ref{nonfilt} in the later sections.

\section{Filtrable rigid rank $2$ bundles on non-K\"ahler surfaces with $a(X)=0$}

The main result of this section is

\begin{thm}\label{main1}
 Let $X$ be a non-K\"ahler compact complex surface with $a(X)=0$. 
 The following statements are equivalent:
\begin{enumerate}
 \item $X$ allows for a filtrable rigid holomorphic rank $2$ vector bundle.
 \item The minimal model of $X$ is an Inoue surface, i.\,e.\,non-K\"ahler, $b_2(X)=0$ and without curves.
\end{enumerate}
Let $X'$ be the minimal model and $\pi:X\seq X'$ the blowdown of all $(-1)$-curves. 
If the statements are valid, then the only filtrable rigid rank $2$ bundle $\tilde{\cal E}$ is the pullback of the
canonical extension on $X'$ via $\pi$, i.\,e.\,$\tilde{\cal E}=\pi^*{\cal E}$ and ${\cal E}$ is the unique rank $2$ bundle
given by the non-split extension
$$0\seq K_{X'}\seq{\cal E}\seq{\cal O}_{X'}\seq 0,$$
up to tensoring with a line bundle. All of these are unstable.
\end{thm}

\begin{proof}
We will use the well known classification of filtrable simple bundles on minimal surfaces with $a(X)=0$. First note that
$X'$ is of class VII by assumption. 
Assume $X$ allows for a filtrable rigid rank $2$ bundle
$\tilde{\cal E}$.
By Proposition \ref{blow} we find a filtrable rigid rank $2$ bundle ${\cal E}$ on $X'$ such that
$\tilde{\cal E}=\pi^*{\cal E}$. This vector bundle ${\cal E}$ is simple and satisfies $\Delta({\cal E})=0$ by Proposition
\ref{imm}. So we can use the classification of those bundles \cite[Thm 4.34]{b95} to obtain $b_2(X')=0$ and $X'$ has no
curves. This is the definition of Inoue surface we use here. 

So it remains to classify filtrable rigid rank $2$ bundles on Inoue surfaces. 
By $b_2(X')=0$ we already know that ${\cal E}$ is simple and topologically trivial. Since ${\cal E}$ is filtrable,
by tensoring with line bundles we may assume that ${\cal E}$ allows for a filtration
$$0\seq{\cal L}\seq{\cal E}\seq{\cal O}_{X'}\seq 0.$$
This is obviously a non-split extension, so
we obtain $H^1(X',{\cal L})\not=0$. On an Inoue surface the intersection form of line
bundles is trivial, hence $\chi({\cal L})=\chi({\cal O}_{X'})=0$; furthermore 
$$H^0({\cal F})\not=0\iff {\cal F}={\cal O}_{X'}$$
for any line bundle, since there are no curves on $X'$. So the condition $H^1({\cal L})\not=0$ translates
into ${\cal L}\in\{K_{X'},{\cal O}_{X'}\}$.

If $L=K_{X'}$ the filtration is 
$$0\seq K_{X'}\seq{\cal E}\seq{\cal O}_{X'}\seq 0.$$
For the non-splitting extension we have
$H^0({\cal E})=0$.
The dual of the filtration implies $H^0({\cal E}^\vee)={\bb C}$, so
if we tensorise this sequence by ${\cal E}$ we obtain
$$0\seq{\cal E}\seq{\cal End}({\cal E})\seq{\cal E}^\vee\seq 0,$$
yielding again that ${\cal E}$ is simple. Tensoring the filtration
and its dual with $K_{X'}$ tells us $H^0(X,{\cal E}\otimes K_{X'})=0=H^0(X,{\cal E}^\vee\otimes K_{X'})$, so also
$$H^2(X,{\cal End}({\cal E}))=H^0(X,{\cal End}({\cal E})\otimes K_{X'})=0.$$
Now we employ
$$0=\chi({\cal End}({\cal E}))=1-h^1(X,{\cal End}({\cal E})).$$
Therefore ${\cal E}$ is rigid.

We are left with the case $L={\cal O}_{X'}$. The non-splitting extension
$$0\seq{\cal O}_{X'}\seq{\cal E}\seq{\cal O}_{X'}\seq 0$$
satisfies ${\cal E}^{\vee}\cong{\cal E}$ and $H^0(X,{\cal E})={\bb C}$. Hence $h^0(X,{\cal End}({\cal E}))\ge 2$, 
so equation \ref{fam} tells us that ${\cal E}$ is not rigid.

Due to Proposition \ref{blow}
it remains to prove the instability of the canonical extension, i.\,e.\,$\deg K_{X}>0$ with respect to any Gauduchon metric. Recall the description
of the various possibilities for Inoue surfaces as given in \cite{i74}. In particular, the universal cover of $X$ is ${\bb C}\times H$, where $H$ denotes 
the upper half plane. It is easy to verify that
$$\Omega_{0}:=-\frac 1{\im w}dz\wedge d\bar z\wedge dw\wedge d\bar w$$
is a volume form on ${\bb C}\times H$ descending to a volume form on an Inoue surface of type $S^{0}$ i.\,e.\,an hermitian metric on $K_{X}^{-1}$.
Its curvature form is
$$\frac {1}{2\pi i}\dd\log\Omega_{0}=-\frac{i}{8\pi(\im w)^{2}}dw\wedge d\bar w.$$
So for any Gauduchon metric $g$ on ${\bb C}\times H$ invariant under $\pi_{1}(X)$ 
one computes on a fundamental domain $D$ of the action of $\pi_{1}(X)$ on ${\bb C}\times H$
$$\deg_{g}K_{X}=-\frac{1}{8\pi}\int_{D} \frac{g_{1\bar 1}}{(\im w)^{2}}dz\wedge d\bar z\wedge dw\wedge d\bar w>0.$$
Similarly, on an Inoue surface of type $S^{+}$ or $S^{-}$ 
$$\Omega:=-\frac{1}{(\im w)^{2}}dz\wedge d\bar z\wedge dw\wedge d\bar w$$
descends to an hermitian metric on $K_{X}^{-1}$ and yields $\deg K_{X}>0$.

\end{proof}

\section{Rigid vector bundles on class VII surfaces}

On a surface with $\chi({\cal O}_X)=0$ and $H^{2}(X,{\bb Z})=0$ the condition $\Delta({\cal E})=0$ simplifies
to $c_2({\cal E})=0$ for any vector bundle ${\cal E}$.
So ${\cal E}$ is simple and topologically trivial, if ${\cal E}$ is rigid,
and the following is just a corollary of Proposition \ref{hopf} below resp. of a result of Moraru \cite[Prop 3.1]{mo04}, stating that there are no simple topologically trivial vector bundles in diagonal Hopf surfaces. 

\begin{cor}\label{nohopf}Let $X$ be a Hopf surface given by diagonal $\varphi$. 
There are no rigid vector bundles on $X$. 
\end{cor}

A strong generalisation of this for rank $2$ vector bundles will be proved in Theorem \ref{main2}.
We would like to emphasise that in contrast to this result, on higher dimensional Hopf manifolds there are rigid
vector bundles. The following gives an example and some characterisation.

\begin{ex}{}\label{exhopf} Let $X$ be the Hopf manifold defined by the automorphism $\phi(z)=2z$ on ${\bb C}^n\setminus\{0\}$. Then there is a natural
smooth elliptic fibration $\pi:X\seq{\bb P}^{n-1}$. For any vector bundle we have $R^1\pi_*\pi^*{\cal E}={\cal E}^\vee$. 
Let ${\cal E}$ be a a simple vector bundle on ${\bb P}^{n-1}$ with
$H^1({\cal End}({\cal E}))=H^2({\cal End}({\cal E}))=0$. Then the Leray spectral sequence 
implies that $H^1({\cal End}(\pi^*{\cal E}))={\bb C}$ and hence $\pi^*{\cal E}$
is rigid. For instance, $T_{{\bb P}^{n-1}}$ satisfies these conditions for $n\ge 2$. (For $n=2$ exactly the line bundles
satisfy the requirements.)
\end{ex}

Moreover, it is known that $p^*T_{{\bb P}^n}$ is not trivial for $n\ge 2$, if $p:{\bb C}^{n+1}\setminus\{0\}\seq{\bb P}^n$ denotes the natural projection. 
The following results will recover this and give a connection between rigid bundles on some Hopf manifolds and non-trivial
vector bundles on ${\bb C}^n\setminus\{0\}$. 

\begin{prop}\label{hopf}Let $X$ be the Hopf manifold given by the quotient of ${\bb C}^n\setminus\{0\}$ by the automorphism group generated by 
$\varphi(z_1,\dots,,z_n)=(\alpha_1z_1,\dots,\alpha_nz_n)$, $|\alpha_i|>1$ and 
$u:{\bb C}^n\setminus\{0\}\seq X$ the projection. 
If ${\cal E}$ is a rigid vector bundle on $X$, then $u^*{\cal E}$ is not trivial.
\end{prop}

\begin{proof}
There is a complex-valued multiplicative degree $\deg_\varphi$ on monomials in $z_1,\dots,z_n$ via
$$\deg_\varphi(z_i):=\alpha_i.$$

If $u^*{\cal E}$ is trivial, usual techniques allow us to identify ${\cal E}$ of rank $r$ with an equivalence class of holomorphic maps
$L:{\bb C^n}\setminus\{0\}\seq Gl(r,{\bb C})$ where
$$L\cong \tilde L:\iff\,\,\exists T\in{\cal O}({\bb C}^n\setminus\{0\}, Gl(r,{\bb C}))\mbox{ such that }
\tilde L=T\circ\varphi\cdot L\cdot T^{-1}.$$
Choosing $T$ carefully we can achieve a normal form of $L$ consisting of blocks $L_\nu, \nu\in{\bb C}^*$ in upper triangle form (cf. \cite{ma92} for a very similar normal form) with the property
$$(L_\nu)_{kk}=\prod_{j=1}^n\alpha_j^{i_{jk}}\nu$$ 
for $i_{jk}\ge 0, i_{j1}=0$ and 
$$(L_\nu)_{kl}\in{\bb C}[z_1,\dots,z_n]$$
homogeneous with respect to $\deg_\varphi$ and  
$$\deg_\varphi(L_\nu)_{kl}=\prod_{j=1}^n\alpha_j^{i_{jl}-i_{jk}}.$$
Now $L(t)_{kl}:=L_{kl}\mbox{ for }(k,l)\notin\{(1,1),(2,2)\}$ and
$$L(t)_{11}:=\exp(t)L_{11}, L(t)_{22}:=\exp(-t)L_{22}$$
defines a non-trivial determinant preserving deformation of ${\cal E}$. 
\end{proof}

Combining (\ref{hopf}) and (\ref{exhopf}) we obtain immediately

\begin{cor}If $n>1, p:{\bb C}^{n+1}\setminus\{0\}\seq{\bb P}^n$ is the natural projection and ${\cal E}$ a simple vector bundle on ${\bb P}^n$ satisfying
$\rk{\cal E}>1$ and $H^1({\cal End}({\cal E}))=H^2({\cal End}({\cal E}))=0$, then $p^*{\cal E}$ is not trivial.
\end{cor}

Let us finish with a consideration of the cases $b_{2}(X)<4$. We denote by $c_{i}^{\bb Q}({\cal E})$ the images of the Chern classes
of ${\cal E}$ in $H^{2i}(X,{\bb Q})=H^{2i}(X,{\bb Z})\otimes_{\bb Z}{\bb Q}$.

\begin{prop}\label{b2sm4}If $X$ is a minimal surface of class VII with 
$b_{2}(X)<4$, then any rigid rank $2$ bundle ${\cal E}$ satisfies $c_{1}^{\bb Q}({\cal E})=0$ and $c_{2}^{\bb Q}({\cal E})=0$,
after tensoring with a line bundle.
\end{prop}

\begin{proof}If $b_{2}(X)=0$, then any vector bundle with $\Delta({\cal E})=0$ is topologically trivial. 
So we assume $b_{2}(X)>0$. 
By Donaldson's first theorem \cite{d87}, the intersection form on $H^{2}(X,{\bb Z})^{\vee\vee}\times H^{2}(X,{\bb Z})^{\vee\vee}\seq{\bb Z}$ is trivial,
i.\,e.\,there is a ${\bb Z}$-basis $e_{i}$ such that
$$e_{i}.e_{j}=-\delta_{ij}.$$
We write also $e_{i}$ for their images in $H^{2}(X,{\bb Q})$ and define integers $\alpha_{i}$ by
$$c_{1}^{\bb Q}({\cal E})=:\sum_{i=1}^{b_{2}(X)}\alpha_{i}e_{i}.$$
Any rigid rank $2$ bundle ${\cal E}$ satisfies $\Delta({\cal E})=0$, so
$\frac 14c_{1}^{\bb Q}({\cal E})^{2}=-c_{2}^{\bb Q}({\cal E})\in{\bb Z}$,
what translates to
$\frac 14|\{i|\,\,\alpha_{i}\text{ odd }\}|\in{\bb Z}$.
Using $b_{2}(X)<4$, we conclude that this number is zero. Now we use $\m{NS}(X)=H^{2}(X,{\bb Z})$, so we
can tensorise with a line bundle ${\cal L}$ such that $c_{1}^{\bb Q}(\cal L)=-\frac 12c_{1}^{\bb Q}({\cal E})$ and obtain $c_{1}^{\bb Q}({\cal E}\otimes{\cal L})=0$ and thus also $c_{2}^{\bb Q}({\cal E}\otimes{\cal L})=0$.
\end{proof}

Proposition \ref{b2sm4} is strong enough to exclude all surfaces allowing for a cycle with $b_{2}(X)<4$. 
Here, by cycle we mean a cycle of rational curves or a smooth elliptic curve.

\begin{thm}\label{main2}On a minimal surface $X$ of class VII with $b_{2}(X)< 4$ allowing for a cycle there is no rigid rank $2$ vector bundle.
\end{thm}

\begin{proof}
The statement will follow by
a theorem of Nakamura if we can prove $C^{2}\le -b_{2}(X)$ and $a(X)=0$. Since the only class VII surfaces with $a(X)=1$ are elliptic
Hopf surfaces and these do not allow for rigid rank $2$ bundles by Corollary \ref{nohopf} we know immediately $a(X)=0$.

We will use the cycle $C$ in a similar manner as Moraru in \cite[Prop 3.1]{mo04}. By Proposition \ref{b2sm4} the vector bundle ${\cal E}|_{C}$ is
topologically trivial. Theorem \ref{main1} tells us that ${\cal E}$ is not filtrable, so any
homomorphism
$$s:{\cal E}\seq{\cal E}\otimes K_{X}\otimes{\cal O}(C)$$
has to vanish if $H^{0}(K_{X}^{2}\otimes{\cal O}(2C))=0$. In this case, $H^{0}({\cal End}({\cal E})\otimes{\cal O}(-C))=0=H^{2}({\cal End}({\cal E})\otimes{\cal O}(-C))$ and Riemann-Roch implies $\chi({\cal End}({\cal E})\otimes{\cal O}(-C))=0$, hence also
$H^{1}({\cal End}({\cal E})\otimes{\cal O}(-C))=0$. The sequence
$$0\seq{\cal End}({\cal E})\otimes{\cal O}(-C)\seq{\cal End}({\cal E})\seq{\cal End}({\cal E})|_{C}\seq 0$$
now implies
$$H^{0}(\m{ad}({\cal E}|_{C}))=0.$$
This, however, contradicts $h^{0}(\m{ad}({\cal F}))\ge 2$ for any topologically trivial rank $2$ bundle ${\cal F}$ on a cycle $C$.

So we know that $K_{X}^{2}\otimes{\cal O}(2C)$ is effective for any cycle $C$. This chain of arguments will be called the cohomological argument
in the sequel of the proof.
We recall (\cite[Lemma 1.2]{te2}) that by minimality of $X$ and negative definiteness of the
intersection form, for any effective divisor $D$ 
$$K_{X}.D\ge 0.$$
We apply this to $K_{X}^{2}\otimes{\cal O}(2C)$ and obtain $K_{X}.C\ge b_{2}(X)$ and subsequently
$$0\ge(K_{X}+C)^{2}\ge b_{2}(X)+C^{2},$$
hence $C^{2}\le -b_{2}(X)$. Theorems of Nakamura \cite[9.2,10.1]{na1} tell us that $X$ must be a half Inoue or a Hopf surface. 

Let us assume that $X$ is half Inoue. Then
there is a unique cycle $C$ and
$$K_{X}^{2}\otimes{\cal O}_{X}(2C)={\cal O}_{X}.$$
Note that our contradiction in the cohomological argument
was not sharp. We will exploit this gap now. Let us denote $R:=K_{X}\otimes{\cal O}_{X}(C)$, the unique $2$-torsion point of
$\m{Pic}^{0}(X)={\bb C}^{*}$, and $s:{\cal E}\seq{\cal E}\otimes R$ a vector bundle homomorphism as above. Taking the determinant tells us
that $s$ has to be a vector bundle isomorphism. Since ${\cal E}$ is simple, there is only one (up to scale), so we obtain
$$H^{2}({\cal End}({\cal E})\otimes{\cal O}_{X}(-C))={\bb C}$$
and the cohomological argument yields $h^{0}(\m{ad}({\cal E}|_{C}))\le 1$, still a contradiction (and this time it is sharp). 

On a Hopf surface, according to Corollary \ref{nohopf} we have only to deal with the case that there is exactly one elliptic curve $C$ and
$K_{X}={\cal O}_{X}(-mC)$ for some $m\ge 1$. Since $K_{X}^{2}\otimes{\cal O}_{X}(2C)$ is effective, we conclude $m=1$. 
So $K_{X}^{-1}={\cal O}_{X}(C)$
is the line bundle associated to a reduced, effective divisor and we obtain similarly by looking at homomorphisms
$$s:{\cal E}\seq{\cal E}\otimes K_{X},$$
by Serre duality and Riemann-Roch (for the first cohomology) 
$$H^{i}(\m{ad}({\cal E}))=H^{i}(\m{ad}({\cal E})\otimes K_{X})=0.$$ 
The rest of the cohomological argument applies and yields a contradiction.
\end{proof}

\begin{rem}
The gap of knowledge for $b_{2}(X)\ge 4$ stems from the lack of techniques to construct non-filtrable bundles. Apart from
pushing down line bundles on a double cover (cp. the result of Aprodu and Toma on elliptic surfaces \cite{at03}) there are
no general methods. Also note that by \cite{tt02} the Chern classes of filtrable rank $2$ bundles are the same as of non-filtrable if
$X$ allows for a cycle.
\end{rem}

\section{Rigid rank $2$ bundles on non-K\"ahler elliptic surfaces}

If $a(X)=1$, then $X$ is an elliptic surface. Here, we do not assume the existence of a section when we say elliptic fibration or elliptic surface.
We refer to \cite[Chapter 2]{bm05} for the results about rank $2$ bundles on elliptic surfaces used here.

\begin{thm}\label{main3}
If a minimal non-K\"ahler elliptic surface $X\stackrel{\pi}{\seq}B$ with smooth fibre $T$ admits a non-filtrable rigid rank $2$ bundle, then 
$X$ is a primary Kodaira surface and $B$ is isogenous to $T^{\vee}$.
\end{thm}

\begin{proof}
It is well-known that any relatively minimal non-K\"ahler elliptic surface $X\stackrel{\pi}{\seq}B$ 
is a quasi-bundle (i.\,e.\,all smooth fibres are isomorphic) with at most multiple fibres
as singular fibres. By going over to a cyclic covering one can get rid of the multiplicity of the fibre and is able to define the notion of a
jump at $b\in B$ as the situation that ${\cal E}|_{X_{b}}$ (after the cyclic covering, if necessary) is unstable. Note that by negative semidefiniteness
of the intersection form any vector bundle on $X$ restricted to a fibre has trivial first Chern class. If ${\cal E}|_{X_{b}}$ is unstable, one obtains
a splitting sequence
$$0\seq\lambda^{\vee}\otimes\det{\cal E}|_{X_{b}}\seq{\cal E}|_{X_{b}}\seq\lambda\seq 0$$
for a line bundle $\lambda\in \m{Pic}(X_{b})$ of degree $-h$ for some $h>0$, called height of the jump. 
This gives rise to an elementary modification $\tilde{\cal E}$ of ${\cal E}$ with discriminant
$$\Delta(\tilde{\cal E})=\Delta({\cal E})-\frac 12h.$$
So we infer that a rigid rank $2$ bundle has no jumps. Thus, the spectral curve $C$ of ${\cal E}$ is an irreducible double cover of $B$ or the
double of a section of the Jacobian. The spectral
curve $C$ is a curve in the Jacobian $J(X)=B\times T^{\vee}\stackrel{p}{\seq}B'$, where $T$ is the generic smooth fibre of $\pi$, with the property that $C\cap p^{-1}(b)$ consists of the direct factors of ${\cal E}|_{X_{b}}$ for generic $b$ and ${\cal E}$. 
The branching locus of $C\seq B$ is computed in the proof
of Lemma 3.10 of \cite{bm05b} (independently of the smoothness of $C$) and has order $8\Delta({\cal E})$. Hence, $C$ is the double of a smooth section
of $J(X)$ or an unramified double cover
of $B$. In the first case, ${\cal E}$ is filtrable, whereas in the second case, ${\cal E}$ is nonfiltrable. In particular, the reduced divisor $C$ is smooth in
any case and has genus $g(C_{red})=g$ in the first case and $g(C)=2g-1$ in the second case. 

In our case we have a smooth curve $C$ of genus $2g-1$.
An important tool in classifying stable rank $2$ bundles is the graph map. Instead of associating $C$ to $E$ we go one step further and note that
$C$ is invariant under the involution $\iota_{\delta_{b}}:\lambda\mapsto\lambda^{\vee}\otimes\delta$ of $\m{Pic}^{0}(X_{b})$ for $\delta_{b}:=\det{\cal E|_{X_{b}}}\in \m{Pic}^{0}(X_{b})$. So,
${\cal E}$ induces a so called graph $A$ as an element of a certain numerical linear system in the ruled surface $F:=J(X)/\iota_{\delta}$. In our case, this is just a section of the ruled surface. The mentioned numerical linear system is denoted by ${\bb P}_{\delta,c_{2}}$. The graph map
$$G:{\cal M}_{\delta,c_{2}}\seq {\bb P}_{\delta,c_{2}}$$
is simply given by mapping a stable rank $2$ vector bundle ${\cal E}$ with $\det{\cal E}=\delta, c_{2}({\cal E})=c_{2}$ to its graph.	

In our case, ${\cal E}$ is stable, since it is non-filtrable. So we are interested in the fibre of the graph map. By \cite[Prop. 4.6]{bm05}
the fibre consists of copies of the Prym variety $\m{Prym}(C/B)$.  The Prym variety has dimension $g(C)-g(B)=g-1$, so we obtain
a determinant preserving deformation of ${\cal E}$ if $g\ge 2$. If $g\le 1$, then $X$ is a Hopf or a Kodaira surface. We already know by Corollary \ref{nohopf} 
that elliptic Hopf surfaces do not allow for rigid rank $2$ bundles.

So let $X$ be a Kodaira surface. If $X$ is a primary Kodaira surface, looking at the projections of $C$ to the factors $B$ and $T^{\vee}$ we see that necessarily $B$ and $T^{\vee}$ must be isogenous.
If $X$ is secondary Kodaira, then $C=L_{1}+L_{2}$ is the sum of two disjoint rational curves what contradicts to the non-filtrabilty of ${\cal E}$.
\end{proof}

\begin{ex}{(A non-filtrable rigid vector bundle on a Kodaira surface)}\label{exkod}
Any double cover $C\hookrightarrow B\times T^{\vee}\seq B$ is invariant under $\iota_{\delta}$ for
any section $\delta$ such that $\delta_{b}$ is the sum of the fibre elements of $C\seq B$ over $b$.
By \cite[Prop. 2.3]{bm05b} any section $\delta$ comes from a line bundle and by \cite[Prop. 4.4, Lemma 3.10]{bm05b} we have a non-filtrable rank $2$ bundle ${\cal E}$ with
$\Delta({\cal E})=0$ exactly if there is a smooth, unramified double cover $C$ of $B$ embeddable into $B\times T^{\vee}$ over $B$, 
where $T$ is the smooth fibre
of $\pi$.
 
Let $B:={\bb C}/\langle 1,p\tau\rangle, C:={\bb C}/\langle 1,2p\tau\rangle$ and $T^{\vee}:={\bb C}/\langle 1,2\tau\rangle$ for an odd integer
$p$ and any complex, non-real number $\tau$. There is a primary Kodaira surface over $B$ with fibre $T$, an unramified double cover $C\stackrel{\delta}{\seq}B$ induced by
the projection from ${\bb C}$, similarly an unramified $p$-fold cover $C\stackrel{\tau}{\seq}T^{\vee}$ and an embedding of $\iota:C\seq B\times T^{\vee}$ over $B$ via $ \iota(x):=(\delta(x),\tau(x))$. This induces a non-filtrable rank $2$ bundle with $\Delta({\cal E})=0$. By Lemma \ref{nonfilt}
${\cal E}$ is rigid.
\end{ex}

So we are left with the case that $C$ is the double of a section of $J(X)$. This means that we have an exact sequence
\begin{equation}\label{2sec}0\seq {\cal O}_X(\sum a_iX_{b_i})\seq{\cal E}\seq\pi^{*}{\cal L}\seq 0\end{equation}
for integers $0\le a_i\le m_i-1$ and some line bundle ${\cal L}$ on $B$, after tensoring ${\cal E}$ with a line bundle; here
the $X_{b_i}$ are the reduced fibres over $b_i$ with multiplicity $m_i$, i.\,e.\,$\pi^*{\cal O}_B(b_i)={\cal O}_X(m_iX_{b_i})$. 
This sequence either splits on every fibre or it does so at most at finitely many fibres. 

\begin{prop}\label{split}For every filtrable simple rank $2$ bundle ${\cal E}$ on $X$ with $\Delta({\cal E})=0$ the sequence (\ref{2sec}) splits at every fibre and $\deg{\cal L}\ge 2-g(B)$.
\end{prop}

\begin{proof}
The simplicity of ${\cal E}$ implies that this sequence does not split and $H^{0}({\cal E}^{\vee})=0$. The second condition implies
$$h^{0}({\cal L}^{\vee})=0.$$
In order to translate the first condition we see that it implies $H^{1}(\pi^{*}{\cal L}^{\vee}\otimes{\cal O}_{X}(\sum a_{i}X_{b_i}))\not=0$. We use a Leray spectral sequence argument in order to push forward
$$h^{1}(\pi^{*}{\cal L}^{\vee}\otimes{\cal O}_{X}(\sum a_{i}X_{b_i}))=h^{1}({\cal L}^{\vee})+h^{0}(R^{1}\pi_{*}(\pi^{*}{\cal L}^{\vee}\otimes{\cal O}_{X}(\sum a_{i}X_{b_i}))).$$
The pushed forward bundle is trivial on the fibres, so $R^{1}\pi_{*}(\pi^{*}{\cal L}^{\vee}\otimes{\cal O}_{X}(\sum a_{i}X_{b_i}))$ is a line bundle and by 
relative duality
$$R^{1}\pi_{*}(\pi^{*}{\cal L}^{\vee}\otimes{\cal O}_{X}(\sum a_{i}X_{b_i}))=(\pi_{*}(\pi^{*}{\cal L}\otimes{\cal O}_{X}(\sum(m_{i}-a_{i}-1)X_{b_i}))))^{\vee}={\cal L}^{\vee}.$$
We already know that ${\cal L}^{\vee}$ has no sections, so the first condition tells us 
$$h^{1}({\cal L}^{\vee})>0.$$
Putting both together, $0>\chi({\cal L}^{\vee})=-\deg{\cal L}+1-g(B)$, i.e.
\begin{equation*}\deg{\cal L}\ge 2-g(B).\end{equation*}

The push forward of sequence (\ref{2sec}) yields
\begin{equation}\label{longpush}0\seq{\cal O}_{B}\stackrel{\mu_{1}}{\seq}\pi_{*}{\cal E}\stackrel{\mu_{2}}{\seq}{\cal L}\stackrel{\mu_{3}}{\seq}{\cal O}_{B}
\stackrel{\mu_{4}}{\seq} R^{1}\pi_{*}{\cal E}\stackrel{\mu_{5}}{\seq}{\cal L}\seq 0.\end{equation}
If sequence (\ref{2sec}) does not split at every fibre then $\pi_{*}{\cal E}$ is a line bundle. On the other hand, $\mu_{3}$ corresponds to a section
$s\in H^{0}({\cal L}^{\vee})=0$, so the long exact sequence (\ref{longpush}) splits into two short exact sequences
\begin{equation}0\seq{\cal O}_{B}\seq\left\{\begin{array}{cc}\pi_{*}{\cal E}\\ R^{1}\pi_{*}{\cal E}\end{array}\right\}\seq{\cal L}\seq 0.\end{equation}
This is not compatible with $\pi_{*}{\cal E}$ being a line bundle.
\end{proof}

\begin{cor}\label{cor3}
There are no filtrable rigid rank $2$ bundles on non-K\"ahler elliptic surfaces $X\seq B$ if $g(B)\ge 1$.
\end{cor}

\begin{proof}We assume there is a rigid filtrable rank $2$ bundle ${\cal E}$.
By the proof of Proposition \ref{split} we have $H^{0}({\cal L}^{\vee})=0, H^{1}({\cal L}^{\vee})\not=0$. By Grauert's semi-continuity theorem,
for every small deformation $\Lambda=({\cal L}_{t})\stackrel{\psi}{\seq}\Delta$ of ${\cal L}$ also $H^{0}({\cal L}_{t}^{\vee})=0$ and thus,
$$h^{1}({\cal L}_{t}^{\vee})=h^{1}({\cal L}^{\vee})$$
and $R^{1}\psi_{*}{\Lambda}^{\vee}$ is a vector bundle of this rank on $\Delta$. It must be trivial, so for the nonzero extension class 
$$\xi\in\Ext^{1}(\pi^{*}{\cal L},{\cal O}_{X}(\sum a_{i}X_{b_{i}}))=\Ext^{1}({\cal L},{\cal O}_{B})=H^{1}({\cal L}^{\vee})$$
(by the Leray and push forward argument of the proof of Prop. \ref{split}) 
given by sequence (\ref{2sec}) there is a nowhere vanishing section
$$\Xi\in H^{0}(\Delta,R^{1}\psi_{*}\Lambda^{\vee})$$
such that $\Xi(0)=\xi$. In turn, this section $\Xi$ gives rise to an extension
\begin{equation}\label{extt}0\seq{\cal O}_{X}(\sum a_{i}X_{b_{i}})\seq\tilde{\cal E}_{t}\seq\pi^{*}{\cal L}_{t}\seq 0\end{equation}
for every $t\in\Delta$. Now we specify the deformation $\Lambda$ to
$${\cal L}_{t}:={\cal L}\otimes {\cal Z}_{t}^{2}$$
for a non-trivial family ${\cal Z}_{t}\in \m{Pic}^{0}(B)$. Here we used $g(B)\ge 1$.  Tensoring (\ref{extt}) with ${\cal Z}_{t}^{-1}$ we obtain
a rank $2$ bundle ${\cal E}_{t}$ as an extension
$$0\seq{\cal O}_{X}(\sum a_{i}X_{b_{i}})\otimes\pi^{*}{\cal Z}_{t}^{-1}\seq{\cal E}_{t}\seq\pi^{*}({\cal L}\otimes{\cal Z}_{t})\seq 0.$$
We want to prove that this deformation is non-trivial.
If $\deg{\cal L}\le 0$, then by pushing forward we see $h^{0}({\cal E}_{t})=0$ for general $t$ and
$h^{0}({\cal E})>0$.

If $\deg{\cal L}>0$, then
by restricting the sequence to a smooth fibre $X_{b}$ we obtain an elementary modification
$$0\seq{\cal O}_{X}(\sum a_{i}X_{b_{i}})\otimes\pi^{*}{\cal Z}_{t}^{-1}\seq{\cal E}_{t}^{(1)}\seq\pi^{*}({\cal L}\otimes{\cal O}(-b)\otimes{\cal Z}_{t})\seq 0.$$
Inductively, we get an elementary modification
$$0\seq{\cal O}_{X}(\sum a_{i}X_{b_{i}})\otimes\pi^{*}{\cal Z}_{t}^{-1}\seq{\cal E}_{t}^{(\deg{\cal L})}\seq\pi^{*}({\cal L}_{0}\otimes{\cal Z}_{t})\seq 0$$
for a line bundle ${\cal L}_{0}\in \m{Pic}^{0}(B)$.  Again it follows $h^{0}({\cal E}_{t}^{(\deg{\cal L})})=0$ for general $t$ and
$h^{0}({\cal E}^{(\deg{\cal L})})>0$.
So the initial deformation was a non-trivial determinant
preserving deformation of ${\cal E}$, contradicting the assumption.
\end{proof}

Finally, we close the last gap. Let us denote ${\cal M}:={\cal O}(\sum b_{i})$.

\begin{cor}\label{cor2}
There are no filtrable rigid rank $2$ bundles on non-K\"ahler elliptic surfaces $X\seq{\bb P}^{1}$.
\end{cor}

\begin{proof}
We may assume $\deg{\cal M}>0$, since otherwise $X$ is an elliptic Hopf surface and the claim is already proved for those. 
We denote $\mu:={\cal O}(\sum X_{b_{i}})={\cal O}(M)$. The ideal
sequence
$$0\seq\m{ad}({\cal E})\otimes\mu^{-1}\seq\m{ad}({\cal E})\seq\m{ad}({\cal E}|_{M})\seq 0$$
implies
\begin{eqnarray*}h^{0}(\m{ad}({\cal E}|_{M}))&\le& h^{1}(\m{ad}({\cal E})\otimes\mu^{-1})\\
&=&h^{2}(\m{ad}({\cal E})\otimes\mu^{-1})\\
&=&h^{0}(\m{ad}({\cal E})\otimes\mu\otimes K_{X})\\
&=&h^{0}(\m{ad}({\cal E})\otimes\pi^{*}(K_{\bb P^{1}}\otimes{\cal M}))\\
&=&h^{0}(\pi_{*}\m{ad}({\cal E})\otimes{\cal O}_{\bb P^{1}}(-2)\otimes{\cal M}).
\end{eqnarray*}
Now we use our assumptions. By Proposition \ref{split} the sequence (\ref{2sec}) splits on every fibre, the left hand side equals $3\deg{\cal M}$ and 
$\pi_{*}\m{ad}({\cal E})$ is a rank $3$ bundle on $B={\bb P}^{1}$ without sections. This means that the right hand side can be
estimated by $3(\deg{\cal M}-2)$, a contradiction.
\end{proof}

\end{document}